\documentclass[preprint]{elsarticle}
\usepackage{array}
\usepackage{amsmath,amsthm,amssymb,amsfonts}						
\usepackage{multirow}
\usepackage{graphicx}
\usepackage{hyperref}
\usepackage{float}
\usepackage{enumitem}
\usepackage{kbordermatrix}									
\usepackage[table]{xcolor}									

\allowdisplaybreaks

\newtheorem{thm}{Theorem}[section]
\newtheorem{lem}[thm]{Lemma}
\newtheorem{cor}[thm]{Corollary}

\theoremstyle{definition}										
\newtheorem{mydef}[thm]{Definition}
\newtheorem{rem}[thm]{Remark}
\newtheorem{ex}[thm]{Example}

\numberwithin{equation}{section}								

\newcommand{\bb}[1]{\mathbb{#1}}								
\newcommand{\ii}{\textup{i}}									
\newcommand{\set}[4]{\left\{ #1 \right\}_{#2 = #3}^{#4}}				

\newcommand{\diag}[1]{\operatorname{\rm diag}\left( #1 \right)}			
\newcommand{\rk}[1]{\operatorname{\rm rank}\left( #1 \right)}				

\newcommand{\sr}[1]{\rho\left(#1\right)}							
\newcommand{\sig}[1]{\sigma \left( #1 \right)}						
\newcommand{\peri}[1]{\pi\left(#1\right)}							
\newcommand{\dg}[1]{\Gamma \left( #1 \right)}						

\newcommand{\mat}[2]{M_{#1}(#2)}								
\newcommand{\jordan}[2]{J_{#1}{\left( #2 \right)}}					

\newcommand{\rinf}[1]{{#1 \rightarrow \infty}}						
\newcommand{\inv}[1]{#1^{-1}}								

\newcommand{\hyp}[2]{\hyperref[#1]{{\rm #2 \ref*{#1}}}}				
\journal{Linear Algebra and its Applications}

\begin{document}
\begin{frontmatter}
\title{Matrix Roots of Imprimitive Irreducible Nonnegative Matrices} 

\author[addy1]{Judith J. McDonald}										
\ead{jmcdonald@math.wsu.edu}
\ead[url]{http://www.math.wsu.edu/math/faculty/jmcdonald/}

\author[addy2]{Pietro Paparella\corref{corpp}}
\ead{pietrop@uw.edu}
\ead[url]{http://www.researchgate.net/profile/Pietro\_Paparella}

\cortext[corpp]{Corresponding author.}

\address[addy1]{Department of Mathematics, Washington State University, Pullman, WA 99164-1113, U.S.A.}
\address[addy2]{Division of Engineering and Mathematics, University of Washington Bothell, Bothell, WA 98011-8246, U.S.A.}

\begin{abstract}
Using matrix function theory, Perron-Frobenius theory, combinatorial matrix theory, and elementary number theory, we characterize, classify, and describe in terms of the Jordan canonical form the matrix \emph{p}th-roots of imprimitive irreducible nonnegative matrices. Preliminary results concerning the matrix roots of reducible matrices are provided as well.
\end{abstract}

\end{frontmatter}

\section{Introduction}

A real matrix $A$ is said to be \emph{eventually nonnegative} (\emph{positive}) if there exists a nonnegative integer $p$ such that $A^k$ is entrywise nonnegative (positive) for all $k \geq p$. If $p$ is the smallest such integer, then $p$ is called the \emph{power index of $A$} and is denoted by $p(A)$.

Eventual nonnegativity  has been the subject of study in several papers (see, e.g., \cite{cnm2002, cnm2004,f1978, jt2004, mpt2014, n2006, nt2009, zm2003, zt1999}) and it is well-known that the notions of eventual positivity and nonnegativity are associated with properties of the eigenspace corresponding to the spectral radius.

A real matrix $A$ is said to possess the \emph{Perron-Frobenius property} if its spectral radius is a positive eigenvalue corresponding to an entrywise nonnegative eigenvector. The \emph{strong Perron-Frobenius property} further requires that the spectral radius is simple; that it dominates in modulus every other eigenvalue of $A$; and that it has an entrywise positive eigenvector.

Several challenges regarding the theory and applications of eventually nonnegative matrices remain unresolved. For example, eventual positivity of $A$ is equivalent to $A$ and $A^\top$ possessing the strong Perron-Frobenius property, however, the Perron-Frobenius property for $A$ and $A^\top$ is a necessary but not sufficient condition for eventual nonnegativity of $A$.

A \emph{matrix pth-root} or \emph{matrix root} of a matrix $A$ is any matrix that satisfies the equation $X^p - A = 0$. An eventually nonnegative (positive) matrix with power index $p=p(A)$ is, by definition, a $p$th-root of the nonnegative (positive) matrix $A^p$. As a consequence, in order to gain more insight into the powers of an eventually nonnegative (positive) matrix, it is natural to examine the roots of matrices that possess the (strong) Perron-Frobenius property. In \cite{mpt2014}, the matrix-roots of eventually positive matrices were classified. In this research, we classify the matrix roots of irreducible imprimitive nonnegative matrices; in particular, the main results in Section 3 provide necessary and sufficient conditions for the existence of an eventually nonnegative matrix $p$th-root of such a matrix. In addition, our proofs demonstrate how to construct these roots given a Jordan canonical form of such a matrix. 

\section{Background}

Denote by $\ii$ the imaginary unit, i.e., $\ii := \sqrt{-1}$. When convenient, an indexed set of the form $\{ x_i, x_{i+1}, \dots, x_{i+j} \}$ is abbreviated to $\left\{ x_k \right\}_{k=i}^{i+j}$. 

For $h$ a positive integer greater than one, let  
\begin{align}
R(h) &:= \{ 0, 1, \dots, h - 1 \}														\nonumber 		\\
\omega =\omega_h &:= \exp{\left( 2 \pi \ii/h \right)} \in \bb{C}, 									\nonumber 		\\
\Omega_h &:= \set{\omega^k}{k}{0}{h-1} = \set{\omega_h^k}{k}{0}{h-1} \subseteq \bb{C}, 				\label{bigomegah}  
\end{align}
and
\begin{align}														
\nu_h := 
\left( 1, \omega, \dots, \omega^{h-1} \right) \in \bb{C}^h.										\label{omegastar} 
\end{align}
With $\omega$ as defined, it is well-known that for $\alpha$, $\beta \in \bb{Z}$,
\begin{align}
\alpha \equiv {\beta} \bmod{h} &\Longrightarrow \omega^\alpha = \omega^\beta  						\label{omalphbeta} 	
\end{align}

Denote by $\mat{n}{\bb{C}}$ ($\mat{n}{\bb{R}}$) the algebra of complex (respectively, real) $n \times n$ matrices. Given $A \in \mat{n}{\bb{C}}$, the \textit{spectrum} of $A$ is denoted by $\sig{A}$, the \emph{spectral radius} of $A$ is denoted by $\sr{A}$, and the \emph{peripheral spectrum}, denoted by $\peri{A}$, is the multi-set given by 
\begin{align*}
\peri{A} = \{ \lambda \in \sig{A} : |\lambda| = \sr{A} \}.
\end{align*} 

 The \emph{direct sum} of the matrices $A_1, \dots, A_k$, where $A_i \in \mat{n_i}{\bb{C}}$, denoted by $A_1 \oplus \dots \oplus A_k$, $\bigoplus_{i=1}^k A_i$, or $\diag{A_1,\dots,A_k}$, is the $n \times n$ matrix 
 \[ 
 \left[
 \begin{array}{ccc}
 A_1 &  & \multirow{2}{*}{\Large 0} \\
 \multirow{2}{*}{\Large 0} & \ddots &  \\
  &  & A_k
 \end{array}
 \right],
 \]
where $n = \sum_{i=1}^k n_i$.

For $\lambda \in \bb{C}$, $\jordan{n}{\lambda}$ denotes the $n \times n$ \emph{Jordan block} with eigenvalue $\lambda$. For $A \in \mat{n}{\bb{C}}$, denote by $J = \inv{Z} A Z = \bigoplus_{i=1}^t \jordan{n_i}{\lambda_i} = \bigoplus_{i=1}^t J_{n_i}$, where $\sum n_i = n$, a Jordan canonical form of $A$. Denote by $\lambda_1,\dots,\lambda_s$ the \textit{distinct} eigenvalues of $A$, and, for $i=1,\dots,s$, let $m_i$ denote the \textit{index} of $\lambda_i$, i.e., the size of the largest Jordan block associated with $\lambda_i$.
 
For $z = r \exp{\left( \ii \theta \right)} \in \bb{C}$, where $r >0$, and an integer $p >1$, let 
\begin{align*} 
z^{1/p} := r^{1/p} \exp{( \ii \theta/p )}, 
\end{align*}
and, for $j \in R(p)$, let
\begin{align}  
f_j (z) := z^{1/p} \exp{(\ii2 \pi j/ p)} = r^{1/p} \exp{\left( \ii \left( \theta + 2 \pi j \right) / p \right)}, 				\label{rtf} 
\end{align}
i.e., $f_j$ denotes the $(j+1)$st-branch of the $p$th-root function.

\subsection{Combinatorial Structure}

For notation and definitions concerning the \emph{combinatorial stucture of a matrix}, i.e., the location of the zero-nonzero entries of a matrix, we follow \cite{br1991} and \cite{h2009}; for further results concerning \emph{combinatorial matrix theory}, see \cite{br1991} and references therein.

A \emph{directed graph} (or simply \emph{digraph}) $\Gamma = (V,E)$ consists of a finite, nonempty set $V$ of \emph{vertices}, together with a set $E \subseteq V \times V$ of \emph{arcs}. For $A \in \mat{n}{\bb{C}}$, the \emph{directed graph} (or simply \emph{digraph}) of $A$, denoted by $\Gamma = \dg{A}$, has vertex set $V = \{ 1, \dots, n \}$ and arc set $E = \{ (i, j) \in V \times V : a_{ij} \neq 0\}$. If $R$, $C \subseteq \{1,\dots, n\}$, then $A[R|C]$ denotes the submatrix of $A$ whose rows and columns are indexed by $R$ and $C$, respectively. 

A digraph $\Gamma$ is \textit{strongly connected} or \textit{strong} if for any two distinct vertices $u$ and $v$ of $\Gamma$, there is a walk in $\Gamma$ from $u$ to $v$ (following \cite{br1991}, we consider every vertex of $V$ as strongly connected to itself). For a strongly connected digraph $\Gamma$, the \emph{index of imprimitivity} is the greatest common divisor of the lengths of the closed walks in $\Gamma$. A strong digraph is \emph{primitive} if its index of imprimitivity is one, otherwise it is \emph{imprimitive}. 

For $n \geq 2$, a matrix $A \in \mat{n}{\bb{C}}$,  is \emph{reducible} if there exists a permutation matrix $P$ such that
\begin{align*}
P^\top A P =
\begin{bmatrix}
A_{11} & A_{12} \\
0 & A_{22}
\end{bmatrix},
\end{align*}
where $A_{11}$ and $A_{22}$ are nonempty square matrices and $0$ is a rectangular zero block. If $A$ is not reducible, then A is called \emph{irreducible}. The connection between reducibility and the digraph of $A$ is as follows: $A$ is irreducible if and only if $\dg{A}$ is strongly connected\footnote{Following \cite{br1991}, vertices are strongly connected to themselves so we take this result to hold for all $n \in \bb{N}$.} (see, e.g., \cite[Theorem 3.2.1]{br1991} or \cite[Theorem 6.2.24]{hj1990}). 

For $h \geq 2$, a digraph $\Gamma = ( V, E )$ is \emph{cyclically $h$-partite} if there exists an ordered partition $\Pi = (\pi_1,\dots, \pi_h)$ of $V$ into $h$ nonempty subsets such that for each arc $(i, j) \in E$, there exists $\ell \in \{ 1, \dots, h \}$ such that $i \in \pi_\ell$ and $j \in \pi_{\ell+1}$ (where, for convenience, we take $V_{h + 1} := V_1$). For $h \geq 2$, a strong digraph $\Gamma$ is cyclically $h$-partite if and only if $h$ divides the index of imprimitivity (see, e.g., \cite[p. 70]{br1991}). A matrix $A \in \mat{n}{\bb{C}}$ is called \emph{h-cyclic} if $\dg{A}$ is cyclically $h$-partite and if $\dg{A}$ is cyclically $h$-partite with ordered partition $\Pi$, then $A$ is said to be \emph{h-cyclic with partition} $\Pi$ or that $\Pi$ \emph{describes} the $h$-cyclic structure of A. The ordered partition $\Pi = (\pi_1,\dots, \pi_h)$ is \emph{consecutive} if $\pi_1 = \{1,\dots, i_1\}$, $\pi_2 = \{i_1 + 1,\dots, i_2\},\dots, \pi_h = \{ i_{h - 1} + 1, \dots, n \}$. If $A$ is $h$-cyclic with consecutive ordered partition $\Pi$, then $A$ has the block form
\begin{align}
\begin{bmatrix} 
0 	& A_{12} 	& 0 		& \cdots 	& 0 		\\
0 	& 0 		& A_{23} 	& \ddots 	& \vdots	\\
\vdots & \vdots 	& \ddots 	& \ddots 	& 0		\\
0 	& 0		& \cdots	& 0 		& A_{(h-1)h}	\\
A_{h1} & 0 & 0 &\cdots & 0
\end{bmatrix}															\label{cyclic_form} 
\end{align}
where $A_{i,i+1} = A[\pi_i|\pi_{i+1}]$ (\cite[p. 71]{br1991}). For any $h$-cyclic matrix $A$, there exists a permutation matrix $P$ such that $P^\top AP$ is $h$-cyclic with consecutive ordered partition. The \emph{cyclic index} or \emph{index of cyclicity} of $A$ is the largest $h$ for which $A$ is $h$-cyclic.

An irreducible nonnegative matrix $A$ is \emph{primitive} if $\dg{A}$ is primitive, and the \emph{index of imprimitivity} of $A$ is the index of imprimitivity of $\dg{A}$.  If $A$ is irreducible and imprimitive with index of imprimitivity $h \geq 2$, then $h$ is the cyclic index of $A$, $\dg{A}$ is cyclically $h$-partite with ordered partition $\Pi = (\pi_1,\dots, \pi_h)$, and the sets $\pi_i$ are uniquely determined (up to cyclic permutation of the $\pi_i$) (see, for example, \cite[p. 70]{br1991}). Furthermore, $\dg{A^h}$ is the disjoint union of $h$ primitive digraphs on the sets of vertices $\pi_i$, $i = 1,\dots, h$ (see, e.g., \cite[\S 3.4]{br1991}).

Following \cite{h2009}, given an ordered partition $\Pi = \left( \pi_1,\dots,\pi_h \right)$ of $\{1, \dots, n\}$ into $h$ nonnempty subsets, the \emph{cyclic characteristic matrix}, denoted by  $\chi_\Pi$, is the $n \times n$ matrix whose $(i,j)$-entry is 1 if there exists $\ell \in \{1,\dots, h\}$ such that $i \in \pi_\ell$ and $j \in \pi_{\ell + 1}$, and 0 otherwise. For an ordered partition $\Pi = \left( \pi_1,\dots,\pi_h \right)$ of $\{1, \dots, n\}$ into $h$ nonnempty subsets, note that 
\begin{enumerate}[label=(\arabic*)]
\item $\chi_\Pi$ is $h$-cyclic and $\dg{\chi_\Pi}$ contains every arc $(i,j)$ for $i \in \pi_\ell$ and $j \in \pi_{\ell+1}$; and
\item $A \in \mat{n}{\bb{C}}$ is $h$-cyclic with ordered partition $\Pi$ if and only if $\dg{A} \subseteq \dg{\chi_\Pi}$.
\end{enumerate}

We recall the Perron-Frobenius Theorem for irreducible, imprimitive matrices.

\begin{thm}[see, e.g., \cite{bp1994} or \cite{hj1990}] \label{pftirr}
Let $A \in \mat{n}{\bb{R}}$, $n \geq 2$, and suppose that $A$ is irreducible, nonnegative, imprimitive and suppose that $h>1$ is the cyclic index of $A$. Then
\begin{enumerate}[label=(\alph*)]
\item $\rho >0$;
\item $\rho \in \sig{A}$;
\item there exists a positive vector $x$ such that $Ax = \rho x$; 
\item $\rho$ is an algebraically (and hence geometrically) simple eigenvalue of $A$; and
\item $\peri{A} = \left\{ \rho \exp{( \ii 2 \pi k/ h )} : k \in R(h) \right\}$.
\item $\omega^k \sig{A}=\sig{A}$ for $k \in R(h)$.
\end{enumerate}
\end{thm}

\subsection{Matrix Functions}

For background material concerning matrix functions, we follow \cite{h2008}; for more results concerning matrix functions, see, e.g., \cite{h2008}, \cite[Chapter 6]{hj1994}, or \cite[Chapter 9]{lt1985}. Herein it is assumed that the complex matrix $A$ has $s \leq t$ distinct eigenvalues.

\begin{mydef} 
Let $f : \bb{C} \longrightarrow \bb{C}$ be a function and let $f^{(k)}$ denote the $k$th derivative of $f$. The function $f$ is said to be \textit{defined on the spectrum of $A$} if the values
\begin{align*}
\begin{array}{c c c}
f^{(k)}(\lambda_i), & k=0,\dots,m_i-1, & i=1,\dots,s,
\end{array}
\end{align*}
called \textit{the values of the function $f$ on the spectrum of $A$}, exist.  
\end{mydef}

\begin{mydef}[Matrix function via Jordan canonical form] 
If $f$ is defined on the spectrum of $A \in \mat{n}{\bb{C}}$, then
\begin{align*} 
f(A) := Z f(J) \inv{Z} = Z \left( \bigoplus_{i=1}^t f(J_{n_i}) \right) \inv{Z}, 
\end{align*}
where
\begin{align}
f(J_{n_i}) := 
\begin{bmatrix} 
f(\lambda_i) & f'(\lambda_i) & \dots   & \frac{f^{(n_i-1)}(\lambda_i)}{(n_i - 1)!} 	\\
	          & f(\lambda_i)  & \ddots & \vdots 						\\
	          &                       & \ddots & f'(\lambda_i)						\\
	          &  	  	   &             & f(\lambda_i)
\end{bmatrix}. 														\label{fjb} 
\end{align}
\end{mydef}

\begin{thm}[see, e.g., {{\cite[\S 1.9]{h2008}} or \cite[Theorem 2.2.5]{pp_d2013}}] \label{jordanformtheorem}
Let $f$ be defined on the spectrum of a nonsingular matrix $A \in \mat{n}{\bb{C}}$ and suppose that $f'(\lambda_i) \neq 0$ for $i=1,\dots,t$. If $J = \bigoplus_{i=1}^t J_{n_i} (\lambda_i) = \inv{Z} A Z$ is a Jordan canonical form of $A$, then 
\begin{align*}
J_f := \bigoplus_{i=1}^t J_{n_i} (f (\lambda_i)) 
\end{align*}
is a Jordan canonical form of $f(A)$.
\end{thm}

\begin{thm}[{\cite[Theorems 2.1 and 2.2]{s2003}}] \label{thm_class_rts} 
If $A \in \mat{n}{\bb{C}}$ is nonsingular, then $A$ has precisely $p^s$ $p$th-roots that are expressible as polynomials in $A$, given by 
\begin{align}
X_j = Z \left( \bigoplus_{i=1}^t f_{j_i} (J_{n_i}) \right) \inv{Z}, 							\label{prim_roots}	
\end{align}
where $j = \begin{pmatrix} j_1, \dots, j_t \end{pmatrix}$, $j_i \in \{0,1,\dots,p-1\}$, and $j_i = j_k$ whenever $\lambda_i = \lambda_k$. 

If $s < t$, then $A$ has additional $p$th-roots that form parameterized families 
\begin{align}
X_j (U) = Z U \left( \bigoplus_{i=1}^t f_{j_i} (J_{n_i}) \right) \inv{U} \inv{Z}, 					\label{nnprim_roots} 
\end{align}
where $U$ is an arbitrary nonsingular matrix that commutes with $J$ and, for each $j$, there exist $i$ and $k$, depending on $j$, such that $\lambda_i = \lambda_k$, while $j_i \neq j_k$.
\end{thm}

In the theory of matrix functions, the roots given by \eqref{prim_roots} are called the {\it primary roots} of $A$, and the roots given by \eqref{nnprim_roots}, which exist only if $A$ is {\it derogatory} (i.e., some eigenvalue appears in more than one Jordan block), are called the {\it nonprimary roots} \cite[Chapter 1]{h2008}.

\subsection{Other results}

\begin{thm}[{\cite[Theorem 13]{p2013}}]\label{omegahfracpwr} 
Let $f_k$ be defined as in \eqref{rtf}. If
\begin{align*}
\mathcal{B} := 
\left\{ j = \left( j_0, j_1, \dots, j_{h-1} \right) : j_k \in R(p), \forall k \in R(h) \right\}, 
\end{align*}
and 
\begin{align*}
(\Omega_h)_j^{1/p} := \set{f_{j_k} (\omega^k)}{k}{0}{h-1}, 
~ j \in \mathcal{B}, 
\end{align*}
then there exists a unique $j \in \mathcal{B}$ such that $(\Omega_h)_j^{1/p} = \Omega_h$ if and only if $\gcd{(h,p)}=1$.
\end{thm}

\begin{rem} \label{remarkrootomega} 
For every $k \in R(h)$ and $j = (j_0,j_1,\dots,j_{h-1}) \in \mathcal{B}$, it is easy to verify that $f_{j_k} ( \omega^k ) = \omega^{(k + h j_k)/p}$,
whence it follows that
\begin{align*}
(\Omega_h)_j^{1/p} 
= \set{\omega^{(k + h j_k)/p}}{k}{0}{h-1}.
\end{align*}

In \cite{p2013} it is shown that if $j$ is the unique $h$-tuple as specified in \hyp{omegahfracpwr}{Theorem}, then the set  
\begin{align*} 
\mathcal{E}_j := \set{\frac{k + h j_k}{p}}{k}{0}{h-1} 
\end{align*}
is integral and a \emph{complete residue system modulo h}, i.e., the map $\phi: \mathcal{E}_j \longrightarrow R(h)$, where 
\[ (k + hj_k)/p \longrightarrow ((k + hj_k)/p) \bmod{h}, \] 
is bijective. Moreover, because $0 \leq (k + h j_k)/p \leq h-1$ (the lower-bound is trivial and the upper bound follows because for any $k$,
\begin{align*}
\frac{k + h j_k}{p} \leq \frac{(h-1) + h(p-1)}{p} = \frac{hp-1}{p} = h - \frac{1}{p} < h)
\end{align*}
it follows that 
\begin{align} 
\mathcal{E}_j = R(h).															\label{fracequalsRh} 
\end{align}

If $( \nu_h )_j^{1/p} := ( f_{j_0}(1), f_{j_1}(\omega), \dots, f_{j_{h-1}}(\omega^{h-1}))$, then, following \eqref{fracequalsRh}, $\left( \nu_h \right)_j^{1/p}$ corresponds to a permutation of $\Omega_h$. 

Because $\gcd{(h,p)}=1$, it follows that $((k + h k j_1)/p) \equiv ((k + h j_k)/p) \bmod{h}$, which, along with \eqref{omalphbeta}, yields 
\begin{align*}
(\omega^{(1 + h j_1)/p})^k = \omega^{(k + h k j_1)/p} = \omega^{(k + h j_k)/p},~k \in \{ 2, \dots, h-1 \}.		
\end{align*}
Hence
\begin{align*}
\left( \nu_h \right)_j^{1/p} 
&= ( 1, \omega^{(1 + h j_1)/p}, \dots, \omega^{((h-1) + hj_{h-1})/p})			\\
&= ( 1, ( \omega^{(1 + h j_1)/p})^1, \dots, ( \omega^{(1 + h j_1)/p})^{h-1}).
\end{align*}
Note that
\begin{align*}
\left( \nu_h \right)^q
:= ( 1, \omega^q, \dots, \omega^{q(h-1)} )
= ( 1, ( \omega^q )^1, \dots, ( \omega^q )^{h-1} ) 	
\end{align*}
is a permutation of the elements in $\Omega_h$ if and only if $\gcd{(h,q)} = 1$ (\cite[Corollary 7]{p2013}); following \eqref{omalphbeta}, for every $h$ there are $\varphi(h)$ such permutations, where $\varphi$ denotes {Euler's \emph{totient} function}. Thus, if $\gcd(h,p) = 1$, then there exists $q \in \bb{N}$, $\gcd{(h,q)} = 1$, such that
\begin{align}
\left( \nu_h \right)_j^{1/p} = \left( \nu_h \right)^q.								\label{fracequalspwr} 
\end{align} 
\end{rem}

Next, we state results concerning the structure of the Jordan chains of $h$-cyclic matrices. It is assumed that $A \in \mat{n}{\bb{C}}$ is nonsingular, $h$-cyclic with ordered-partition $\Pi$, and has the form \eqref{cyclic_form}.

\begin{cor} [{\cite[Corollary 3.2]{mp2014}}] \label{jblockcor}
If $\jordan{r}{\lambda}$ is a Jordan block of $J$, then $\jordan{r}{\lambda \omega^k}$ is a Jordan block of $J$ for $k \in R(h)$.
\end{cor}

\begin{rem} 
With $\nu_h$ as defined in \eqref{omegastar} and
\begin{align}
J \left( \lambda \nu_h, r \right) := 
\begin{bmatrix}
\jordan{r}{\lambda} 				\\
& \jordan{r}{\lambda \omega} 			\\
& & \ddots 						\\
& & & \jordan{r}{\lambda \omega^{h-1}}
\end{bmatrix} \in \mat{rh}{\bb{C}},										\label{cyclicjordanblocks} 
\end{align}
it follows that a Jordan form of a nonsingular $h$-cylic matrix $A$ has the form
\begin{align*}
\inv{Z} A Z = J = \bigoplus_{i=1}^{t'} J\left( \lambda_i \nu_h, r_i \right),~t'|t.
\end{align*}
\end{rem}

\begin{lem}[{\cite[Theorem 3.6]{mp2014}}] \label{cycliccommutator}
For $i=1,\dots,t'$, if 
\begin{align}
A_{\lambda_i} :=
Z
\diag{0, \dots, 0, \overbrace{J\left( \lambda_i \nu_h, r_i \right)}^i,0,\dots,0}
\inv{Z} \in \mat{n}{\bb{C}},											\label{cycliccommutatormatrix} 
\end{align}
then $\dg{A_{\lambda_i}} \subseteq \dg{\chi_\Pi}$, $A_{\lambda_i}$ commutes with $A$, and $A_{\lambda_i} A_{\lambda_j} = A_{\lambda_j} A_{\lambda_i} = 0$ for $i \neq j$, $j=1,\dots,t'$.
\end{lem}

\begin{cor} \label{cycliccor}
If $x$ is a strictly nonzero right eigenvector and $y$ is a strictly nonzero left eigenvector of $A$ corresponding to $\lambda \in \bb{C}$, then $A_\lambda$ has cyclic index $h$ and $\dg{A_\lambda} = \dg{\chi_\Pi}$.
\end{cor}

In particular, we examine the case when $A$ is a nonnegative, irreducible, imprimitive, nonsingular matrix with index of cyclicity $h$. Without loss of generality, it is assumed that $\sr{A} = 1$. Following \hyp{pftirr}{Theorem} and \hyp{jblockcor}{Corollary}, note that the Jordan form of $A$ is
\begin{align*}
\inv{Z} A Z = 
\begin{bmatrix}
J(\nu_h,1)	& 							\\
			& J\left( \lambda_2 \nu_h,r_2 \right) & &	\\
			& & \ddots 							\\
			& & & J\left( \lambda_t \nu_h, r_{t'} \right)
\end{bmatrix}.													
\end{align*}
Consider the matrix 
\begin{align*}
A_1 := Z \begin{bmatrix} J(\nu_h,1) & 0 \\ 0 & 0 \end{bmatrix} \inv{Z}. 
\end{align*}
Following \hyp{cycliccor}{Corollary}, $\dg{A_1} = \dg{\chi_\Pi}$, $AA_1 = A_1 A$, and, following \hyp{pftirr}{Theorem}, there exist positive vectors $x$ and $y$ such that $Ax = x$ and $y^\top A = y^\top$. If $x$ and $y$ are partitioned conformably with $A$ as
\begin{align*}
x =
\begin{bmatrix}
x_1 	\\
x_2 	\\
\vdots \\
x_h
\end{bmatrix} \text{and }
y^\top = 
\begin{bmatrix}
y_1^\top & y_2^\top & \cdots & y_h^\top 
\end{bmatrix}, 
\end{align*}
then  
\begin{align*}
A_1 = h
\begin{bmatrix}
0 & x_1 y_2^\top & \cdots & \cdots & 0 		\\
0 & 0 & x_2 y_3^\top & \cdots & 0 			\\
\vdots & \vdots & \ddots & \ddots  & \vdots	\\
0 & 0 & \cdots &  0 & x_{h-1} y_h^\top		\\
x_h y_1^\top & 0 &  \cdots & 0 & 0
\end{bmatrix} \geq 0.
\end{align*}

\section{Main Results}

Unless otherwise noted, we assume $A \in \mat{n}{\bb{R}}$ is a nonnegative, nonsingular, irreducible, imprimitive matrix with $\sr{A} = 1$. Before we state our main results, we introduce additional concepts and notation: following Friedland \cite{f1978}, for a multi-set $\sigma = \set{\lambda_i}{i}{1}{n} \subseteq \bb{C}$, let $\sr{\sigma} := \max_i \{ | \lambda_i | \}$ and $\bar{\sigma} := \set{\bar{\lambda}_i}{i}{1}{n} \subseteq \bb{C}$. If $\sigma = \bar{\sigma}$, we say that $\sigma$ is \emph{self-conjugate}. Clearly, $\sigma$ is self-conjugate if and only if $\bar{\sigma}$ is self-conjugate.

The (multi-)set $\sigma$ is said to be a \emph{Frobenius (multi)-set} if, for some positive integer $h \leq n$, the following properties hold:
\begin{enumerate}[label = (\roman*)]
\item $\sr{\sigma} > 0$;
\item $\sigma \cap \{ z \in \bb{C}: |z| = \sr{\sigma }\} = \sr{\sigma } \Omega_h$; and
\item $\sigma = \omega \sigma$, i.e., $\sigma$ is invariant under rotation by the angle $2\pi/h$. 
\end{enumerate} 
Clearly, the set $\Omega_h$ as defined in \eqref{bigomegah} is a self-conjugate Frobenius set.

The importance of Frobenius multi-sets becomes clear in view of the following result, which was introduced and stated without proof in \cite[\S 4, Lemma 1]{f1978} and proven rigorously in \cite[Theorem 3.1]{zt1999}.

\begin{lem} \label{friedlandlemma}
Let $A$ be an eventually nonnegative matrix. If $A$ is not nilpotent, then the spectrum of $A$ is a union of self-conjugate Frobenius sets.
\end{lem} 

Let $\lambda = r \exp{(\ii \theta)} \in \bb{C}$, $\Im{\left( \lambda \right)} \neq 0$, $\lambda \Omega_h := \left\{ \lambda, \lambda \omega_h, \dots, \lambda \omega_h^{h-1} \right\}$, $\varphi = 2 \pi k/h$, and assume $\gcd{(h,p)=1}$. With $f_j$ as defined in \eqref{rtf}, a tedious but straightforward calculation shows that
\begin{align}
f_i (\lambda) f_{j_k} ( \omega^k )= f_{j_k^{(i)}} ( \hat{\lambda} ), 							\label{lambdaanalysis} 
\end{align}
where $\hat{\lambda} = r \exp{( \ii ( \theta + \varphi ) )}$ and $j_k^{(i)} = {\left( i + j_k \right)}\bmod{p} \in \{ 0,1, \dots, p-1 \}$ (\cite[pp.~62--63]{pp_d2013}).

Hence, following \hyp{omegahfracpwr}{Theorem}, there exists a unique $j \in \mathcal{B}$ such that, for all $i \in R(p)$, the set 
\begin{align}
f_i (\lambda) \left( \Omega_h \right)_j^{1/p}
:= \left\{ f_i (\lambda) f_{j_0} \left( 1 \right), f_i (\lambda) f_{j_1} \left( \omega_h \right), \dots, f_i (\lambda) f_{j_{h-1}} \left( \omega_h^{h-1} \right) \right\}																\label{lambdaanalysis2} 
\end{align}
is a self-conjugate Frobenius set; moreover, following \eqref{lambdaanalysis}, there exists $j^{(i)} = \left( j_0^{(i)}, j_1^{(i)}, \dots, j_{h-1}^{(i)} \right) \in \mathcal{B}$ such that 
\begin{align}
f_i (\lambda) \left( \Omega_h \right)_j^{1/p}
= \left\{ f_{j_0^{(i)}} \left(\lambda\right) , f_{j_1^{(i)}} \left( \lambda \omega_h \right), \dots, f_{j_{h-1}^{(i)}} \left( \lambda \omega_h^{h-1} \right) \right\}. 																\label{lambdaanalysis3} 
\end{align}

As a consequence of \hyp{omegahfracpwr}{Theorem} and \hyp{friedlandlemma}{Lemma}, it should be clear that $A$ can not possess an eventually nonnegative $p$th-root if $\gcd(h,p) > 1$; however, more can be ascertained.

Partition the Jordan form of $A$ as 
\begin{align}
\begin{bmatrix}
J_+ 				\\
& J_- 				\\
& & J_\bb{C}
\end{bmatrix}															\label{jordformirr} 
\end{align}
where: 
\begin{enumerate}[label=(\arabic*)]
\item $J\left( \lambda \nu_h, r \right)$ is a submatrix of $J_+$ if and only if $\sig{J\left( \lambda \nu_h, r \right)} \cap \bb{R}^+ \neq \emptyset$ and $\sig{J\left( \lambda \nu_h, r \right)} \cap \bb{R}^- = \emptyset$;
\item $J\left( \lambda \nu_h, r \right)$ is a submatrix of $J_-$ if and only if $\sig{J\left( \lambda \nu_h, r \right)} \cap \bb{R}^- \neq \emptyset$; and
\item $J\left( \lambda \nu_h, r \right)$ is a submatrix of $J_\bb{C}$ if and only if $\sig{J \left( \lambda \nu_h, r \right)} \cap \bb{R} = \emptyset$.
\item $J \left( \lambda \nu_h, r \right)$ is defined as in \eqref{cyclicjordanblocks}, i.e.,
\begin{align*}
J \left( \lambda \nu_h, r \right) := 
\begin{bmatrix}
\jordan{r}{\lambda} 				\\
& \jordan{r}{\lambda \omega} 			\\
& & \ddots 						\\
& & & \jordan{r}{\lambda \omega^{h-1}}
\end{bmatrix} \in \mat{rh}{\bb{C}},
\end{align*}
\end{enumerate}
Suppose $J_+$ has $r_1$ distinct blocks of the form $J\left( \lambda \nu_h, r \right)$, $J_-$ has $r_2$ distinct blocks of the form $J\left( \lambda \nu_h, r \right)$, and $J_\bb{C}$ has $c$ distinct blocks of the form $J\left( \lambda \nu_h, r \right)$. 

We are now ready to present our main results.

\begin{thm} \label{mainresultirr}
Let $A \in \mat{n}{\bb{R}}$, and suppose that $A \geq 0$, nonsingular, irreducible, and imprimitive. Let $h > 1$ be the cyclic index of $A$, let $\Pi$ describe the $h$-cyclic structure of $A$, let the Jordan form of $A$ be partitioned as in \eqref{jordformirr}, and suppose that $\gcd(h,p)=1$. 

If $p$ is even, and
\begin{enumerate}[label=(\alph*)]
\item $r_2 = 0$, then $A$ has $2^{r_1-1} p^c$ eventually nonnegative primary $p$th-roots; or
\item $r_2 >0$, then $A$ has  no eventually nonnegative primary $p$th-roots. 
\end{enumerate}

If $p$ is odd, then $A$ has $p^c$ eventually nonnegative primary $p$th-roots.
\end{thm}

\begin{proof} 
For $\lambda \in \bb{C}$, $\lambda \neq 0$, and $j = \left( j_0, j_1, \dots, j_{h-1} \right) \in \mathcal{B}$, let
\begin{align*}
F_j \left( J\left( \lambda \nu_h, r \right) \right) := 
\begin{bmatrix}
f_{j_0} \left( \jordan{r}{\lambda} \right)				\\
& f_{j_1} \left( \jordan{r}{\lambda \omega} \right) 			\\
& & \ddots 									\\
& & & f_{j_{h-1}} \left( \jordan{r}{\lambda \omega^{h-1}} \right)
\end{bmatrix}.
\end{align*}
We form an eventually nonnegative root $X$ by carefully selecting a root for each submatrix of every block appearing in \eqref{jordformirr}. 

Case 1: \emph{$p$ is even}. We consider the blocks appearing in \eqref{jordformirr}:
\begin{enumerate}[label=(\emph{\roman*})]
\item \label{item1} $J_+$: Following \hyp{pftirr}{Theorem}, $J\left( \nu_h, 1 \right)$ is a submatrix of $J_+$ (note that $h$ must be odd) and, following \hyp{omegahfracpwr}{Theorem}, there exists $j = (j_0, j_1, \dots, j_{h-1}) \in \mathcal{B}$ such that $\sig{F_j \left( J\left( \nu_h, 1 \right) \right)} = \Omega_h$. With that specific choice of $j$, it is clear that 
\begin{align*}
\sr{F_j \left( J\left( \nu_h, 1 \right) \right)} = 1. 
\end{align*}

For any other submatrix $J\left( \lambda \nu_h, r \right)$ of $J_+$, without loss of generality, we may assume that $\lambda \in \bb{R}^+$. For every such $\lambda$, $k$ must be chosen such that $f_k (\lambda)$ is real (see \cite[Corollary 2.16]{mpt2014}) and $f_k (\lambda)$ is real if $k=0$ or $k=p/2$; hence, there are two choices such that $f_k (\lambda)$ is real and, following \eqref{lambdaanalysis2} and \eqref{lambdaanalysis3}, there are two choices such that $\sig{F_{j^{(k)}} \left( J\left( \lambda \nu_h, r \right) \right)}$ is a self-conjugate Frobenius set.

\item $J_-$: If $r_2 > 0$, then $A$ does not have a real primary root so that, \emph{a fortiori}, it can not have an eventually nonnegative primary root. 

\item \label{item3} $J_\bb{C}$: For any submatrix $J\left( \lambda \nu_h, r \right)$ of $J_\bb{C}$, following \eqref{lambdaanalysis2} and \eqref{lambdaanalysis3}, for the same $j$ chosen for $F_j \left( J\left( \nu_h, 1 \right) \right)$ in \hyperref[item1]{\ref*{item1}}, there exists \begin{align*}
j^{(k)} = \left( j_0^{(k)}, j_1^{(k)}, \dots, j_{h-1}^{(k)} \right) \in \mathcal{B}
\end{align*}
such that $\sig{F_{j^{(k)}} \left( J\left( \lambda \nu_h, r \right) \right)}$
is a self-conjugate Frobenius set for all $k \in R(p)$. Thus, there are $p^c$ ways to choose roots for blocks in $J_\bb{C}$. 
\end{enumerate}
Following the analysis contained in \hyperref[item1]{\ref*{item1}}--\hyperref[item3]{\ref*{item3}}, note that there are $2^{r_1 - 1} p^c$ ways to form a root in this manner.

Case 2: \emph{$p$ is odd}. Following \hyp{omegahfracpwr}{Theorem} and properties of the $p${th}-root function, if $p$ is odd, then for any submatrix $J\left( \lambda \nu_h, r  \right)$ of $J_+$ or $J_-$, there is only one choice $j \in \mathcal{B}$ such that $\sig{F_j \left( J\left( \lambda \nu_h, r \right) \right)}$ is a self-conjugate Frobenius set. For submatrices $J\left( \lambda \nu_h, r  \right)$ of $J_\bb{C}$, the analysis is the same as in \hyperref[item3]{\ref*{item3}}. In this manner, there are $p^c$ possible selections.

Partition the Jordan form of $A$ as
\begin{align*}
\begin{bmatrix}
J(\nu_h,1) & \\
& \hat{J}
\end{bmatrix}.
\end{align*} 
With the above partition in mind, consider the matrix $p$th-root of $A$ given by
\begin{align*}
X = 
Z
\begin{bmatrix}
F_j \left( J(\nu_h,1) \right) & \\
 & F (\hat{J})
\end{bmatrix}
\inv{Z}
\end{align*}
where $j$ is selected as in \hyp{omegahfracpwr}{Theorem} and $F(\hat{J})$ is a $p$th-root of $\hat{J}$ containing blocks of the form $F_{j^{(k)}} \left( J\left( \lambda \nu_h, r \right) \right)$, chosen as in Case 1 or Case 2. 

The matrix $X$ must be irreducible because every power of a reducible matrix is reducible. If $\bar{h}$ is the cyclic index of $X$, then $1 \leq \bar{h} \leq h$ (if $\bar{h} > h$, then $X$ would have $\bar{h}$ eigenvalues of maximum modulus and consequently so would $A$, contradicting the maximality of $h$) and $\bar{h}$ must divide $h$. However, we claim that $\bar{h} = h$. For contradiction, assume $\bar{h} < h$ and consider the matrix $A_1$ given by 
\begin{align*}
A_1 =
Z
\begin{bmatrix}
J(\nu_h,1) & 0 \\
0 & 0
\end{bmatrix}
\inv{Z}.
\end{align*}
Following \hyp{cycliccor}{Corollary}, $A_1 \geq 0$, irreducible, $h$-cyclic, and $\Pi$ describes the $h$-cyclic structure of $A_1$. Next, consider the matrix $X_1$ given by
\begin{align*}
X_1 =
Z
\begin{bmatrix}
F_j \left( J(\nu_h,1) \right) & 0 \\
0 & 0
\end{bmatrix}
\inv{Z},
\end{align*}
which is a matrix $p$th-root of $A_1$. Following \hyp{cycliccor}{Corollary}, the cyclic index of $X_1$ is $\bar{h}$. However, following the remarks leading up to \eqref{fracequalspwr}, there exists $q \in \bb{N}$ relatively prime to $h$ such that $X_1 = A_1^q \geq 0$ which, along with $X_1$ being a $p$th-root of $A_1$, implies $X_1 = X_1^{pq}$. Thus, $X_1$ is a nonnegative, irreducible, $h$-cyclic matrix, contradicting the maximality of $\bar{h}$. Hence, $\bar{h} = h$ and $X$ is $h$-cyclic. Before we continue with the proof, note that $\dg{X} \subseteq \dg{X_1} = \dg{\chi_\Pi^q}$.

Following \hyp{jordanformtheorem}{Theorem}, there exists $\bar{Z} \in \mat{n}{\bb{C}}$ such that 
\begin{align*}
&\inv{\bar{Z}} X \bar{Z} = \\
&\begin{bmatrix}
F_j \left( J(\nu_h,1) \right)	& 										\\
			& J\left( f_{j_2}(\lambda_2) \left( \nu_h \right)_j^{1/p},r_2 \right) & &	\\
			& & \ddots 												\\
			& & & J\left( f_{j_t}(\lambda_t) \left( \nu_h \right)_j^{1/p}, r_{t'} \right)
\end{bmatrix}.
\end{align*} 
By construction of $\bar{Z}$ (\cite[Lemma 1.3.2]{pp_d2013}), note that 
\begin{align*}
X_1 =
Z
\begin{bmatrix}
F_j \left( J(\nu_h,1) \right) & 0 \\
0 & 0
\end{bmatrix}
\inv{Z}
= 
\bar{Z} 
\begin{bmatrix}
F_j \left( J(\nu_h,1) \right) & 0 \\
0 & 0
\end{bmatrix} 
\inv{\bar{Z}}.
\end{align*}
Let 
\begin{align*}
X_2 :=
\bar{Z}
\begin{bmatrix}
0	& 										\\
			& J\left( f_{j_2}(\lambda_2) \left( \nu_h \right)_j^{1/p},r_2 \right) & &	\\
			& & \ddots 												\\
			& & & J\left( f_{j_t}(\lambda_t) \left( \nu_h \right)_j^{1/p}, r_{t'} \right)
\end{bmatrix}
\inv{\bar{Z}}.
\end{align*}
Following \hyp{cycliccommutator}{Lemma}, $X_2$ is $h$-cyclic, $\dg{X_2} \subseteq \dg{\chi_\Pi^q}$, and $X_1 X_2 = X_2 X_1 = 0$. Thus, for $k \in \bb{N}$, $X^k = X_1^k + X_2^k$. Because $\sr{X_2} < 1$, $\lim_\rinf{k} X_2^k = 0$. The matrices $X_1$ and $X_2$ are $h$-cyclic, $\rk{X_1^2} = \rk{X_1}$, $\rk{X_2^2} = \rk{X_2}$, and $\dg{X_2} \subseteq \dg{X_1} = \dg{\chi_\Pi^q}$, thus, following \cite[Theorem 2.7]{h2009}, note that $\dg{X_2^k} \subseteq \dg{X_1^k} = \dg{\chi_\Pi^{qk}}$. Thus, there exists $p \in \bb{N}$ such that  $X_1^k > X_2^k$ for all $k \geq p$, i.e., $X$ is eventually nonnegative.
\end{proof}

\begin{ex}
We demonstrate \hyp{mainresultirr}{Theorem} via an example: consider the matrix
\begin{align*}
A =
\frac{1}{3}
\begin{bmatrix}
0 & 0 & 2 & 1 & 0 & 0 	\\
0 & 0 & 1 & 2 & 0 & 0	\\
0 & 0 & 0 & 0 & 2 & 1	\\
0 & 0 & 0 & 0 & 1 & 2	\\
2 & 1 & 0 & 0 & 0 & 0	\\
1 & 2 & 0 & 0 & 0 & 0
\end{bmatrix} \in \mat{6}{\bb{R}}.
\end{align*}
One can verify that $A = Z D \inv{Z}$, where
\begin{align*}
Z = \left[
\begin{array}{*{6}{r}}
1 & 1 & 1 & 1 & 1 & 1						\\
1 & 1 & 1 & -1 & -1 & -1						\\
1 & \omega & \omega^2 &  1 & \omega & \omega^2		\\
1 & \omega & \omega^2 & -1 &  -\omega &  -\omega^2	\\
1 & \omega^2 & \omega &  1 & \omega^2 & \omega		\\
1 & \omega^2 & \omega & -1 &  -\omega^2 &  -\omega		
\end{array}
\right],
\end{align*} 
and
\begin{align*}
D = \diag{1,\omega,\omega^2,\frac{1}{3},\frac{1}{3}\omega,\frac{1}{3}\omega^2}.
\end{align*}
Because $A$ has six distinct eigenvalues, following \hyp{thm_class_rts}{Theorem}, it has $2^6 = 64$ primary square roots (and no nonprimary roots).

The matrices
\begin{align*} 
\hat{D} = \diag{1,\omega^2,\omega,\frac{\sqrt{3}}{3},\frac{\sqrt{3}}{3}\omega^2,\frac{\sqrt{3}}{3}\omega}
\end{align*}
and
\begin{align*} 
\tilde{D} = \diag{1,\omega^2,\omega,-\frac{\sqrt{3}}{3},-\frac{\sqrt{3}}{3}\omega^2,-\frac{\sqrt{3}}{3}\omega}
\end{align*}
are square roots of $D$, and, following \hyp{mainresultirr}{Theorem}, the matrices
\begin{align*}
\hat{X}= Z \hat{D} \inv{Z} \approx 
\left[
\begin{array}{*{6}{r}}
0 & 0 & 0 & 0 & 0.7887 & 0.2113 \\
0 & 0 & 0 & 0 & 0.2113 & 0.7887 \\
0.7887 & 0.2113 & 0 & 0 & 0 & 0 \\         
0.2113 & 0.7887 & 0 & 0 & 0 & 0 \\      
0 & 0 & 0.7887 & 0.2113 & 0 & 0	 \\    
0 & 0 & 0.2113 & 0.7887 & 0 & 0  
\end{array}
\right]
\end{align*}
and
\begin{align*}
\tilde{X}= Z \tilde{D} \inv{Z} \approx 
\left[
\begin{array}{*{6}{r}}
0 & 0 & 0 & 0 & 0.2113 & 0.7887 \\
0 & 0 & 0 & 0 & 0.7887 & 0.2113 \\
0.2113 & 0.7887 & 0 & 0 & 0 & 0 \\         
0.7887 & 0.2113 & 0 & 0 & 0 & 0	\\      
0 & 0 & 0.2113 & 0.7887 & 0 & 0 \\    
0 & 0 & 0.7887 & 0.2113 & 0 & 0   
\end{array}
\right]
\end{align*}
are the only eventually nonnegative square roots of $A$.
\end{ex}

\begin{rem} \label{resultderogmatrices}
Following the notation of \hyp{mainresultirr}{Theorem}, note that all primary roots of $A$ are given by 
\begin{align*}
X =
Z
\begin{bmatrix}
F_{j_1} \left( J(\nu_h,1) \right)	& 							\\
			& F_{j_2} \left( J\left( \lambda_2 \nu_h, r_2 \right) \right) & &	\\
			& & \ddots 										\\
			& & & F_{j_t} \left( J\left( \lambda_t \nu_h, r_{t'} \right) \right)
\end{bmatrix}
\inv{Z}
\end{align*}
where $j_k \in \mathcal{B}$ for $k = 1,\dots, t'$, and subject to the constraint that $j_k = j_\ell$ (because $j_k$ and $j_\ell$ are ordered $h$-tuples, equality is meant entrywise) if $\lambda_k = \lambda_\ell$.

If $A$ is \emph{deregatory}, i.e., some eigenvalue appears in more than one Jordan block in the Jordan form of $A$, then $A$ has additional roots of the form 
\begin{align*}
X(U) =
ZU
\begin{bmatrix}
F_{j_1} \left( J(\nu_h,1) \right)	& 							\\
			& F_{j_2} \left( J\left( \lambda_2 \nu_h, r_2 \right) \right) & &	\\
			& & \ddots 										\\
			& & & F_{j_t} \left( J\left( \lambda_t \nu_h, r_{t'} \right) \right)
\end{bmatrix}
\inv{U} \inv{Z}
\end{align*}
where $U$ is any matrix that commutes with $J$ and subject to the constraint that $j_k \neq j_\ell$, if $\lambda_k = \lambda_\ell$. Select branches $j_1,j_2, \dots, j_t$ following the proof of \hyp{mainresultirr}{Theorem}. Following \cite[Theorem 1, \S 12.4]{lt1985}, note that the matrix $U$ must be of the form
\begin{align*}
\begin{bmatrix}
u_1 	& 					\\
 	& u_2 	& 				\\
 	& 	& \ddots 			\\
 	& 	& 		& u_h 		\\
 	& 	& 	& 	& \hat{U} 
\end{bmatrix}
\end{align*}
and $X(U)$ is real if $U$ is selected to be real. The matrix $X(U)$ is irreducible because $X(U)^p = A$. Moreover,
\begin{align*}
X_1 (U) 
&:= Z U
\begin{bmatrix}
F_{j_1} \left( J(\nu_h,1) \right) & 0 \\
0 & 0
\end{bmatrix}
\inv{U} \inv{Z} 
= X_1,						
\end{align*}
where $X_1$ is defined as in the proof of \hyp{mainresultirr}{Theorem}. Thus, $X_1(U)$ is a nonnegative, irreducible, $h$-cyclic matrix that commutes with $X(U)$, so that the argument demonstrating the eventual nonnegativity of $X$ is also valid for $X(U)$.
\end{rem}

\begin{cor} \label{realrootcor}
Let $A \in \mat{n}{\bb{R}}$ and suppose that $A \geq 0$, irreducible, and imprimitive with index of cyclicity $h$. Then $A$ possesses an eventually nonnegative $p$th-root if and only if $A$ possesses a real $p$th-root and $\gcd{(h,p)} = 1$.  
\end{cor}

\begin{thm} \label{thmforgeneralirrennmatrices}
Let $A \in \mat{n}{\bb{R}}$ and suppose that $A \geq 0$, irreducible, and imprimitive. If $\inv{Z} A Z = J = J_0 \oplus J_1$ is a Jordan canonical form of $A$, where $J_0$ collects all the singular Jordan blocks and $J_1$ collects the remaining Jordan blocks, and $A$ possesses a real root, then all eventually nonnegative $p$th-roots of $A$ are given by $A = Z \left( X_0 \oplus X_1 \right) \inv{Z}$, where $X_1$ is any $p$th-root of $J_1$ characterized by \hyp{mainresultirr}{Theorem} or \hyp{resultderogmatrices}{Remark}, and $X_0$ is a real $p$th-root of $J_0$.
\end{thm}

Although the following result is known (see \cite[Algorithm 3.1 and its proof]{h2009}), our work provides another proof.

\begin{cor} \label{evennongen}
If $A \in \mat{n}{\bb{R}}$ is irreducible with index of cyclicity $h$, where $h>1$, and $\Pi$ describes the $h$-cyclic structure of $A$, then $A$ is eventually nonnegative if and only if there exists a nonsingular matrix $Z$ such that
\begin{align*}
\inv{Z} A Z =
\begin{bmatrix}
J( \nu_h, 1) &  \\
 & \hat{J}  
\end{bmatrix},
\end{align*} 
and, associated with $\sr{A} = 1 \in \sig{A}$, is a positive left eigenvector $x$ and right eigenvector $y$. 
\end{cor}

\begin{proof}
If $A$ is eventually nonnegative, select $p$ relatively prime to $h$ such that $A^p \geq 0$. Then $A$ is a $p$th-root of $A^p$ and the result follows from \hyp{thmforgeneralirrennmatrices}{Theorem}.

The converse is clear by setting 
\[ X_1 = Z \begin{bmatrix} J( \nu_h, 1) & 0 \\ 0 & 0 \end{bmatrix} \inv{Z} \] 
and 
\[ X_2 = Z \begin{bmatrix} 0 & 0 \\ 0 & \hat{J} \end{bmatrix} \inv{Z}. \] 
\end{proof}

\begin{rem}
\hyp{thmforgeneralirrennmatrices}{Theorem} remains true if the assumption of nonnegativity is replaced with eventual nonnegativity.
\end{rem}

\begin{rem}
For $A \in \mat{n}{\bb{C}}$ with no eigenvalues on $\bb{R}^-$, the {\it principal} $p$th-root, denoted by $A^{1/p}$, is the unique $p$th-root of $A$ all of whose eigenvalues lie in the segment $\{ z : -\pi/p < \arg(z) < \pi/p\}$ \cite[Theorem 7.2]{h2008}. A nonnegative matrix $A$ is \emph{stochastic} if $\sum_j a_{ij} = 1$ for all $i=1,\dots,n$. Following \hyp{omegahfracpwr}{Theorem}, the principal $p$th-root of an imprimitve irreducible stochastic matrix is never stochastic. 
\end{rem}

\subsection{Reducible Matrices}

Identifying the eventually nonnegative matrix roots of nonnegative reducible matrices poses many obstacles, chief of which is controlling the entries of the off-diagonal blocks in the Frobenius normal form. Moreover, the assumption that $\gcd{(h,p)} = 1$ is not necessarily required; for instance, consider the matrix
\begin{align*}
A =
\begin{bmatrix}
0 & 1 & 0 & 0 \\
0 & 0 & 1 & 0 \\
0 & 0 & 0 & 1 \\
1 & 0 & 0 & 0
\end{bmatrix}
\end{align*}
and note that the matrix
\begin{align*}
B: =
\begin{bmatrix}
0 & 0 & 1 & 0 \\
0 & 0 & 0 & 1 \\
1 & 0 & 0 & 0 \\
0 & 1 & 0 & 0
\end{bmatrix} = A^2
\end{align*}
is a reducible 2-cyclic matrix; $\sig{B} = \Omega_2 \cup \Omega_2$; $B$ obviously possesses an irreducible nonnegative square root; and $\gcd{(2,2)} = 2 > 1$.

\begin{mydef}
A matrix $A \in \mat{n}{\bb{C}}$ is said to be \emph{completely reducible} if there exists a permutation matrix $P$ such that 
\begin{align}
P^\top A P 
= \bigoplus_{i=1}^k A_{ii}
= 
 \left[
 \begin{array}{ccc}
 A_{11} &  & \multirow{2}{*}{\large 0} \\
 \multirow{2}{*}{\large 0} & \ddots &  \\
  &  & A_{kk}
 \end{array}
 \right],												\label{compreducibleform}
\end{align}	
where $k \geq 2$ and $A_{11}, \dots, A_{kk}$ are square irreducible matrices.									
\end{mydef} 

\begin{rem}
Following the definition, if $A$ is completely reducible, then, without loss of generaltiy, we may assume $A$ is in the form of the matrix on the right-hand side of \eqref{compreducibleform}. Furthermore, it is clear that $A$ is eventually nonnegative if and only if $A_{11}, \dots, A_{kk}$ are eventually nonnegative. 
\end{rem}

The following is corollary to \hyp{realrootcor}{Corollary} and \hyp{thmforgeneralirrennmatrices}{Theorem} .

\begin{cor}
Let $A$ be eventually nonnegative, nonsingular, and completely reducible. Let $h_i$ denote the cyclicity of $A_{ii}$ for $i=1,\dots,k$. If each $A_{ii}$ possesses a real $p$th-root, then $A$ possesses an eventually nonnegative $p$th-root if and only if $\gcd{(h_i,p)} = 1$ for all $i$.
\end{cor}


\bibliographystyle{abbrv}
\bibliography{laabib,crs}

\end{document}